\documentclass[12pt]{article}

\usepackage{latexsym,amssymb,upref,amsmath,amsthm, amsfonts,authblk}
\usepackage{amssymb,amsmath,amsthm, calc, graphicx}
\usepackage{epsfig}
\usepackage{breqn}
\usepackage{footnpag}
\usepackage{rotating}
\usepackage{xcolor}
\usepackage{amsfonts}
\usepackage{setspace}
\usepackage{fullpage}
\usepackage{enumitem}
\usepackage{bbold}
\usepackage{comment}
\usepackage{pgf,tikz}
\usepackage{mathrsfs}
\usetikzlibrary{arrows}
\usepackage{yhmath}
\usepackage{hyperref}
\usepackage{authblk}

\newtheorem*{thm*}{Theorem*}
\newtheorem*{prop*}{Proposition}

\newtheorem{thm}{Theorem}

\newtheorem{claim}{Claim}

\newcommand{\thistheoremname}{}
\newtheorem*{genericthm*}{\thistheoremname}
\newenvironment{namedthm*}[1]
{\renewcommand{\thistheoremname}{#1}%
	\begin{genericthm*}}
	{\end{genericthm*}}

\newcommand{\abs}[1]{\left\lvert{#1}\right\rvert}

\title{Triangles in $C_5$-free graphs \\ and Hypergraphs of Girth Six}

\author
{
Beka Ergemlidze
\thanks{ Department of Mathematics, Central European University, Budapest.
		E-mail: \texttt{beka.ergemlidze@gmail.com}} \qquad 
Abhishek Methuku \thanks{\'Ecole polytechnique f\'ed\'erale de Lausanne and Central European University, Budapest. (Corresponding) E-mail: \texttt{abhishekmethuku@gmail.com}}
}

\begin{document}

\maketitle

\begin{abstract}
We introduce a new approach and prove that the maximum number of triangles in a $C_5$-free graph on $n$ vertices is at most $$(1 + o(1)) \frac{1}{3 \sqrt 2} n^{3/2}.$$ 

We also show a connection to $r$-uniform hypergraphs without (Berge) cycles of length less than six, and estimate their maximum possible size.

\end{abstract}

\section{Introduction}
Motivated by a conjecture of Erd\H os  \cite{Erd} on the maximum possible number of pentagons in a triangle-free graph, Bollob\'as and Gy\H ori \cite{BolGy} initiated the study of the natural converse of this problem. Let $ex(n,K_3,C_5)$ denote the maximum possible number of triangles in a graph on $n$ vertices without containing a cycle of length five as a subgraph. Bollob\'as and Gy\H ori \cite{BolGy} showed that 

\begin{equation}
\label{eq:BGY}
(1 + o(1)) \frac{1}{3 \sqrt 3}  n^{3/2} \le ex(n,K_3,C_5) \le  (1 + o(1)) \frac{5}{4} n^{3/2}.
\end{equation}

Their lower bound comes from the following example: Take a $C_4$-free bipartite graph $G_0$ on $n/3 + n/3$ vertices with about $(n/3)^{3/2}$ edges and double each vertex in one of the color classes and add an edge joining the old and the new copy to produce a graph $G$. Then, it is easy to check that $G$ contains no $C_5$ and it has $(n/3)^{3/2}$ triangles.

Recently, F\"uredi and \"Ozkahya \cite{Furedi} gave a simpler proof showing a slighly weaker upper bound of $\sqrt{3}n^{3/2} + O(n)$. Alon and Shikhelman \cite{AlonS} improved these results by showing that
\begin{equation}
\label{eq:AS}
ex(n,K_3,C_5) \le (1 + o(1)) \frac{\sqrt 3}{2}  n^{3/2}.
\end{equation}

Ergemlidze, Gy\H{o}ri, Methuku and Salia \cite{ErgGMS}  recently showed that 
\begin{equation}
\label{eq:EGMS}
ex(n,K_3,C_5) \le (1 + o(1)) \frac{1}{2\sqrt 2}  n^{3/2}.
\end{equation}

 In this paper our aim is to introduce a new approach and use it to improve two old results and prove a new one. Our approach consists of carefully counting paths of length $5$ (or paths of length $3$) by making use of the structure of certain subgraphs. Roughly speaking, we are able to efficiently bound the number of $5$-paths if its middle edge lies in a dense subgraph (for e.g., in a $K_4$). We expect this approach to have further applications.

Our first result improves the previous estimates \eqref{eq:BGY}, \eqref{eq:AS}, \eqref{eq:EGMS}, on the maximum possible number of triangles in a $C_5$-free graph, as follows.

\begin{thm}
	\label{Main_Result}
	We have,
\begin{displaymath}
ex(n,K_3,C_5) < (1 + o(1)) \frac{1}{3 \sqrt 2}  n^{3/2}.
\end{displaymath}
\end{thm}

Given a hypergraph $H$, its \emph{2-shadow} is the graph consisting of the edges $\{ab \mid ab \subset e \in E(H)\}$. Applying our approach to the $2$-shadow of a hypergraph of girth $6$, we prove the following result.
 
\begin{thm}
\label{girth6}
Let $H$ be an $r$-uniform hypergraph of girth $6$. Then $$\abs{E(H)} \le (1+o(1)) \frac{n^{3/2}}{r^{3/2}(r-1)}.$$
\end{thm}

Let us mention a related result of Lazebnik and Verstra\"ete \cite{Lazeb_Verstraete} which states the following. If $H$ is an $r$-uniform hypergraph of girth $5$, then $$\abs{E(H)} \le (1+o(1)) \frac{n^{3/2}}{r(r-1)}.$$
Note that Theorem \ref{girth6} shows that if a (Berge) cycle of length $5$ is also forbidden, then the above bound can be improved by a factor of $\sqrt{r}$.

In Section \ref{FurtherImprovementsection}, we show a close connection between Theorem \ref{Main_Result} and Theorem \ref{girth6}, and prove that the estimate in Theorem \ref{Main_Result} can be slightly improved using Theorem \ref{girth6}. However, to illustrate the main ideas of the proof of Theorem \ref{Main_Result}, we decided to state Theorem \ref{Main_Result} in a slightly weaker form.
 
Loh, Tait, Timmons and Zhou \cite{Loh_Tait_Timmons} introduced the problem of simultaneously forbidding an induced copy of a graph and a (not necessarily induced) copy of another graph. A graph is called induced-$F$-free if it does not contain an induced copy of $F$.
They asked the following question: What is the largest size of an induced-$C_4$-free and $C_5$-free graph on $n$ vertices? 
They noted that the example showing the lower bound in \eqref{eq:BGY} is in fact induced-$C_4$-free and $C_5$-free, thus it gives a lower bound of $(1+o(1)) \frac{2}{3 \sqrt{3}} n^{3/2}$. (If the ``induced-$C_4$-free" condition is replaced by ``$C_4$-free" condition, then Erd\H{o}s and Simonovits \cite{Erd_Sim} showed that the answer is $(1+o(1))\frac{1}{2 \sqrt{2}} n^{3/2}$.) This question seems to be difficult to answer. In \cite{EGM}, Gy\H{o}ri and the current authors determined (asymptotically) the maximum size of an induced-$K_{s,t}$-free and $C_{2k+1}$-free graph on $n$ vertices in all the cases except in the case when $s=t=2$ and $k=2$ (which is the above question) but in this case an upper bound of only $n^{3/2}/2$ was proven \cite{EGM}. Here we show that using our approach one can slightly improve this upper bound.

\begin{thm}
\label{indC4C5}
If a graph $G$ is $C_5$-free and induced-$C_4$-free, then $$\abs{E(G)} \le (1+o(1)) \frac{n^{3/2}}{2\sqrt[10]{2}}.$$
\end{thm}



\vspace{2mm}
\textbf{Structure of the paper:} In Section \ref{MainproofSection}, we prove Theorem \ref{Main_Result}. In Section \ref{Girth6plusImprovementSection}, we prove Theorem \ref{girth6} and show how it gives a slight improvement in Theorem \ref{Main_Result}. Finally in Section \ref{indC4C5section}, we prove Theorem \ref{indC4C5}.

\vspace{2mm}
\textbf{Notation:}
Given a graph $G$ and a vertex $v$ of $G$, let $N_1(v)$ and $N_2(v)$ denote the first neighborhood and the second neighborhood of $v$ respectively. 

For a vertex $v$ of $G$, let $d(v)$ be the degree of $v$. The average degree of a graph $G$ is denoted by $d(G)$, or simply $d$ if it is clear from the context. The maximum degree of a graph $G$ is denoted by $d_{max}(G)$ or simply $d_{max}$.

A \emph{walk} or \emph{path} usually referes to an unordered one, unless specified otherwise. That is, a walk or path $v_1v_1v_2 \ldots v_k$ is considered equivalent to $v_kv_{k-1}v_2 \ldots v_1$.

\section{Number of triangles in a $C_5$-free graph: Proof of Theorem \ref{Main_Result}}
\label{MainproofSection}

Let $G$ be a $C_5$-free graph with maximum possible number of triangles. We may assume that each edge of $G$ is contained in a triangle, because otherwise, we can delete it without changing the number of triangles. Two triangles $T, T'$ are said to be in the same \emph{block} if they either share an edge or if there is a sequence of triangles $T, T_1, T_2, \ldots, T_s, T'$ where each triangle of this sequence shares an edge with the previous one (except the first one of course). It is easy to see that all the triangles in $G$ are partitioned uniquely into blocks. Notice that any two blocks of $G$ are edge-disjoint.  Below we will characterize the blocks of $G$.

A block of the form $\{abc_1, abc_2, \ldots, abc_k\}$ where $k \ge 1$, is called a \emph{crown-block} (i.e., a collection of triangles containing the same edge) and a block consisting of all triangles contained in the complete graph $K_4$ is called a \emph{$K_4$-block}. See Figure \ref{figure1}.

\begin{figure}[h]
	\begin{center}
		\includegraphics[scale=0.25]{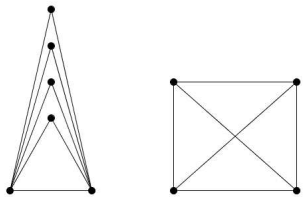}
	\end{center}
	\caption{An example of a crown-block and a $K_4$-block}
	\label{figure1}
\end{figure}

The following claim was proved in \cite{ErgGMS}. We repeat its proof for completeness.

\begin{claim}
	\label{Blocks}
Every block of $G$ is either a crown-block or a $K_4$-block.
\end{claim}
\begin{proof}
If a block contains only one or two triangles, then it is easy to see that it is a crown-block. So we may assume that a block of $G$ contains at least three triangles and let $abc_1, abc_2$ be some two triangles in it. We claim that if $bc_1x$ or $ac_1x$ is a triangle in $G$ which is different from $abc_1$, then $x = c_2$. Indeed, if $x \not = c_2$, then the vertices $a,x,c_1,b,c_2$ contain a $C_5$, a contradiction. Similarly, if $bc_2x$ or $ac_2x$ is a triangle in $G$ which is different from $abc_2$, then $x = c_1$. 

Therefore, if $ac_i$ or $bc_i$ (for $i = 1, 2$) is contained in two triangles, then $abc_1c_2$ forms a $K_4$. However, then there is no triangle in $G$ which shares an edge with this $K_4$ and is not contained in it because if there is such a triangle, then it is easy to find a $C_5$ in $G$, a contradiction. So in this case, the block is a $K_4$-block, and we are done. 

So we can assume that whenever $abc_1, abc_2$ are two triangles then the edges $ac_1, bc_1, ac_2, bc_2$ are each contained in exactly one triangle. Therefore, any other triangle which shares an edge with either $abc_1$ or $abc_2$ must contain $ab$. Let $abc_3$ be such a triangle. Then applying the same argument as before for the triangles $abc_1, abc_3$ one can conclude that the edges $ac_3, bc_3$ are contained in exactly one triangle and so, any other triangle of $G$ which shares an edge with one of the triangles $abc_1, abc_2, abc_3$ must contain $ab$ again. So by induction, it is easy to see that all of the triangles in this block must contain $ab$. Therefore, it is a crown-block, as needed.
\end{proof}

\textbf{Edge Decomposition of $G$:} We define a decomposition $\mathcal D$ of the edges of $G$ into paths of length 2, triangles and $K_4$'s, as follows: Since each edge of $G$ belongs to a triangle, and all the triangles of $G$ are partitioned into blocks, it follows that the edges of $G$ are partitioned into blocks as well. Moreover, by Claim \ref{Blocks}, edges of $G$ can be decomposed into crown-blocks and $K_4$-blocks. We further partition the edges of each crown-block $\{abc_1, abc_2, \ldots, abc_k\}$  (for some $k \ge 1$) into the triangle $abc_1$ and paths $ac_ib$ where $2 \le i \le k$. This gives the desired decomposition $\mathcal D$ of $E(G)$.

\begin{claim}
	\label{atmosttwo2paths}
Let $u, v$ be two non-adjacent vertices of $G$. Then the number of paths of length $2$ between $u$ and $v$ is at most two. Moreover, if $uxv$ and $uyv$ are the paths of length $2$ between $u$ and $v$, then $x$ and $y$ are adjacent.
\end{claim}
\begin{proof}
First let us prove the second part of the claim. Since we assumed every edge is contained in a triangle and $u$ and $v$ are not adjacent, there is a vertex $w\not =v$ such that $uxw$ is a triangle. If $w \not = y$, then $uwxvy$ is a $C_5$, a contradiction. So $w = y$, so $x$ and $y$ are adjacent, as desired.  

Now suppose that there are $3$ distinct vertices $x,y,z$ such that $uxv, uyv, uzv$ are paths of length $2$ between $u$ and $v$. Then $x$ and $y$ are adjacent by the discussion in the previous paragraph. Therefore  $uxyvz$ is a $C_5$ in $G$, a contradiction, proving the claim.
\end{proof}

Let $t(v)$ be the number of triangles containing a vertex $v$ and let $t(G) = t = \sum_{v\in V(G)} \frac{t(v)}{n}$. 
Observe that number of triangles in $G$ is $nt/3$. Our goal is to bound $t$ from above.

First we claim that for any vertex $v$ of $G$, 
\begin{equation}
\label{ConnectionBetweendandt}
t(v) \leq d(v)\leq 2 t(v).
\end{equation}

Indeed, $d(v)\leq 2 t(v)$ simply follows by noting that every edge is in a triangle. Now notice that $t(v)$ is equal to the number of edges contained in the first neighborhood of $v$ (denoted by $N_1(v)$). Moreover, there is no path of length three in the subgraph induced by $N_1(v)$ because otherwise there is a $C_5$ in $G$. So by Erd\H{o}s-Gallai theorem, the number of edges contained in $N_1(v)$ is at most $\frac{3-1}{2}\abs{N_1(v)} = d(v)$. Therefore, $t(v) \leq d(v)$.

Note that by adding up \eqref{ConnectionBetweendandt} for all the vertices $v \in V(G)$ and dividing by $n$, we get 
\begin{equation}
\label{ConnectionBetweendaveandtave}
t \leq d\leq 2 t.
\end{equation}

Suppose there is a vertex $v$ of $G$, such that $t(v) < t/3$. Then we may delete $v$ and all the edges incident to $v$ from $G$ to obtain a graph $G'$ such that $t(G') > 3(nt/3-t/3)/(n-1) = t(G)$. Then it is easy to see that if the theorem holds for $G'$, then it holds for $G$ as well. Repeating this procedure, we may assume that for every vertex $v$ of $G$, $t(v) \ge t/3$. Therefore, by \eqref{ConnectionBetweendandt}, we may assume that the degree of every vertex of $G$ is at least $t/3$.

\begin{claim}
\label{max_degree}
We may assume that $d_{max}(G) \le 6 \sqrt{3} \sqrt{n}$.
\end{claim}
\begin{proof}
Suppose that there is a vertex $v$ such that $d(v) >  6 \sqrt{3} \sqrt{n}$. The sum of degrees of the vertices in $N_1(v)$ is at least $\frac{\abs{N_1(v)}t}{3} = \frac{d(v)t}{3}$ as we assumed that the degree of every vertex is at least $t/3$. The number of edges inside $N_1(v)$ is $t(v)$, which is at most $d(v)$ by \eqref{ConnectionBetweendandt}. Therefore the number of edges between $N_1(v)$ and $N_2(v)$ is at least $\frac{d(v)t}{3} - 2d(v)$. Now notice that any vertex in $N_2(v)$ is incident to at most two of these edges by Claim \ref{atmosttwo2paths}. Therefore, $\abs{N_2(v)} \ge \frac{d(v)t}{6} - d(v)$. 

Thus we have,  $$n \ge \abs{N_1(v)} + \abs{N_2(v)} \geq d(v) + \frac{d(v)t}{6} - d(v) = \frac{d(v)t}{6} > \frac{6 \sqrt{3} \sqrt{n}t}{6}, $$
which implies $t < \sqrt{\frac{n}{3}}$. Therefore, the total number of triangles in $G$ is less than $\frac{n^{3/2}}{3 \sqrt{3}}$, proving Theorem \ref{Main_Result}.
\end{proof}

By the Blakley-Roy inequality, the number of (unordered) walks of length five in $G$ is $nd^5/2$. First let us show that most of these walks are paths. Let $v_0v_1v_2v_3v_4v_5$ be a walk that is not a path. Then $v_i = v_j$ for some $i < j$. Fix some $i < j$. Then there are $n$ choices for $v_0$, and then at most $d_{max}$ choices for every $v_k$ with $k \le j-1$, then since $v_j = v_i$, there is only choice for $v_j$ and again at most $d_{max}$ choices for every $v_k$ with $k \ge j+1$. So in total the number of walks that are not paths is at most $\binom{6}{2} n(d_{max})^4$ as there are $\binom{6}{2} = 15$ choices for $i, j$. Thus the number of (unordered) paths of length five in $G$ is at least $nd^5/2 - 15 n(d_{max})^4$. From now, we refer to a path of length five as a $5$-path. 

We say a $5$-path $v_0v_1v_2v_3v_4v_5$ is \emph{bad} if there exists an $i$ such that $v_iv_{i+1}v_{i+2}$ is a triangle of $G$; otherwise it called \emph{good}. Our aim is to show that the number of bad $5$-paths is very small. Let $v_0v_1v_2v_3v_4v_5$ be a bad $5$-path. Then there is an $i$ so that $v_iv_{i+1}v_{i+2}$ is a triangle. If we fix an $i$, there are at most $2nt$ choices for $v_iv_{i+1}v_{i+2}$ as each of the $nt/3$ triangles can be ordered in $3! = 6$ ways, and there are at most $d_{max}$ choices for every vertex $v_k$ with $k < i$ or $k > i+2$. There are four choices for $i$. Therefore, the total number of $5$-paths that are bad is at most $8nt (d_{max})^3$. This means that the number of good $5$-paths is at least $nd^5/2 - 15 n(d_{max})^4-8nt (d_{max})^3$. By \eqref{eq:BGY}, the number of triangles of $G$ is at most $(1+o(1))\frac{5n^{3/2}}{4}$. Since the number of triangles of $G$ is $nt/3$, we have $t \le \frac{15}{4}(1+o(1))n^{1/2}$. Now using Claim \ref{max_degree}, it follows that the number of good $5$-paths is at least 
\begin{equation}
\label{lower_bound}
\frac{nd^5}{2} - 15 n(6 \sqrt{3} \sqrt{n})^4- 8n\frac{15}{4} n^{1/2} (6 \sqrt{3} \sqrt{n})^3 \ge \frac{nd^5}{2}  - C n^3,
\end{equation}
where $C$ is some positive constant.

Now we seek to bound the number of good $5$-paths from above. Recall that we defined a decomposition $\mathcal D$ of the edges of $G$ into three types of subgraphs: paths of length 2, triangles and $K_4$'s.
We distingush three cases depending on which type of subgraph the middle edge of a good $5$-path belongs to, and bound the number of good $5$-paths in each of those cases separately in the following three claims. 

A path of length two (or a $2$-path) $xyz$ is called \emph{good} if $x$ and $z$ are not adjacent. 

\begin{claim}
\label{Claim2path}
Let $abc$ be a $2$-path of the edge-decomposition $\mathcal D$. Then the number of good $5$-paths in $G$ whose middle edge is either $ab$ or $bc$ is at most $n^2$. 
\end{claim}

\begin{proof}
A good $5$-path $xypqzw$ whose middle edge is $ab$ or $bc$ contains good $2$-paths, $xyp, qzw$ as subpaths (where $pq$ is either $ab$ or $bc$). Moreover, since $xypqzw$ is a good $5$-path and the $2$-path $abc$ is contained in the triangle $abc$ (because of the way we defined the decomposition $\mathcal D$), it follows that $x,y \not \in \{a,b,c\}$ and $z,w \not \in \{a,b,c\}$. 

Let $n_a$ be the number of good $2$-paths in $G$ of the form $axy$ where $x, y \not \in \{a,b,c\}$, and let $n_b$ be the number of good $2$-paths in $G$ of the form $bxy$ where $x, y \not \in \{a,b,c\}$. We define $n_c$ similarly. 
Then the number of good $5$-paths whose middle edge is either $ab$ or $bc$ is at most $$ n_an_b + n_b n_c = n_b (n_a + n_c) \le \left( \frac{n_a + n_b + n_c}{2} \right)^2.$$

We claim that for any fixed vertex $y \not \in \{a,b,c\}$, there are at most two good $2$-paths of the form $pxy$ with $p \in \{a,b,c\}$ and $x \not \in \{a,b,c\}$. If this claim is true, then $n_a + n_b + n_c \le 2n$, so the right-hand-side of the above inequality is at most $n^2$, proving Claim \ref{Claim2path}.

It remains to prove this claim. Suppose for a contradiction that there are three such good $2$-paths, say, $p_1x_1y, p_2x_2y, p_3x_3y$. Notice that if $p_ix_i$ is disjoint from $p_jx_j$ for some $i, j \in \{1,2,3\}$, then $p_ip_jx_jyx_i$ forms a $C_5$ in $G$, a contradiction (note that here we used that $p_i$ and $p_j$ are adjacent even when $\{p_i, p_j\} = \{a,c\}$ because of the way we defined $\mathcal D$). Thus the edges $p_1x_1, p_2x_2, p_3x_3$ pair-wise intersect, which implies that either $p_1 = p_2 = p_3 = p$ or $x_1 = x_2 = x_3 = x$ (since $p_1, p_2, p_3 \in \{a,b,c\}$ and $x_1, x_2, x_3 \not \in \{a,b,c\}$). The former case is impossible by Claim \ref{atmosttwo2paths} and in the latter case, note that $a,b,c,x$ forms a $K_4$, but this contradicts the definition of $\mathcal D$ since $abc$ was assumed to be a $2$-path component of $\mathcal D$ and no $2$-path of $\mathcal D$ comes from a $K_4$-block of $G$.
\end{proof}

\begin{claim}
\label{Triangle}
Let $abc$ be a triangle of the edge-decomposition $\mathcal D$. Then the number of good $5$-paths in $G$ whose middle edge is either $ab, bc, ca$ is at most $\frac{4n^2}{3}$.
\end{claim}
\begin{proof}
The proof is very similar to that of the proof of Claim \ref{Claim2path}. A good $5$-path $xypqzw$ whose middle edge is $ab, bc, ca$ contains good $2$-paths, $xyp, qzw$ as subpaths. Moreover, since $xypqzw$ is a good $5$-path, it follows that $x,y \not \in \{a,b,c\}$ and $z,w \not \in \{a,b,c\}$.

Let $n_a$ be the number of good $2$-paths in $G$ of the form $axy$ where $x, y \not \in \{a,b,c\}$, and let $n_b, n_c$ be defined similarly. 
Then the number of good $5$-paths whose middle edge is $ab, bc$ or $ca$ is at most $$ n_an_b + n_b n_c + n_cn_a \le  \frac{(n_a + n_b + n_c)^2}{3}.$$

By the same argument as in the proof of Claim \ref{Claim2path}, it is easy to see that $n_a + n_b + n_c \le 2n$, so the above inequality finishes the proof.
\end{proof}

\begin{claim}
\label{pathsinK_4}
Let $abcd$ be a $K_4$ of the edge-decompostion $\mathcal D$. Then the number of good $5$-paths in $G$ whose middle edge belongs to the $K_4$ is at most $\frac{3n^2}{2}$.
\end{claim}
\begin{proof}
Notice that any good $5$-path $xypqzw$ contains good $2$-paths, $xyp, qzw$ as subpaths. Suppose the middle edge of $xypqzw$ belongs to the $K_4$, $abcd$. Then since $xypqzw$ is a good $5$-path, it follows that $x,y \not \in \{a,b,c,d\}$ and $z,w \not \in \{a,b,c,d\}$.

Let $n_a$ be the number of good $2$-paths in $G$ of the form $axy$ where $x, y \not \in \{a,b,c,d\}$, and let $n_b, n_c, n_d$ be defined similarly. Then the number of good $5$-paths whose middle edge belongs to the $K_4$, $abcd$ is at most \begin{equation} \label{k4ninj} \sum_{i,j \in \{a,b,c,d\}} n_i n_j \le \frac{3}{8} (n_a +n_b +n_c + n_d)^2. \end{equation}
To see that the above inequality is true one simply needs to expand and rearrange the inequality $ \sum_{i,j \in \{a,b,c,d\}} (n_i - n_j)^2 \ge 0.$ 

Using a similar argument as in the proof of Claim \ref{Claim2path}, it is easy to see that for any fixed vertex $y \not \in \{a,b,c, d\}$, there are at most two good $2$-paths of the form $pxy$ with $p \in \{a,b,c, d\}$ and $x \not \in \{a,b,c,d\}$. This implies that $n_a +n_b +n_c + n_d \le 2n$, so using \eqref{k4ninj}, the proof is complete.
\end{proof}

Now we are ready to bound the number of good $5$-paths in $G$ from above. Suppose the number of edges of $G$ is $e(G)$, and let $\alpha_1 e(G)$ and $\alpha_2 e(G)$ be the number of edges of $G$ that are contained in triangles and $2$-paths of the edge-decomposition $\mathcal D$ of $G$, respectively. Let $\alpha_1 + \alpha_2 = \alpha$. In other words, $(1-\alpha) e(G)$ edges of $G$ belong to the $K_4$'s in $\mathcal D$. Then the number of triangles and $2$-paths in $\mathcal D$ is at most $\frac{\alpha_1}{3}e(G)$ and $\frac{\alpha_2}{2}e(G)$ respectively and the number of $K_4$'s in $\mathcal D$ is at most $\frac{(1-\alpha)}{6}e(G)$. Therefore, using Claim \ref{Claim2path}, Claim \ref{Triangle} and Claim \ref{pathsinK_4}, the total number of good $5$-paths in $G$ is at most $$\frac{\alpha_1}{3}e(G) \frac{4n^2}{3} + \frac{\alpha_2}{2}e(G)n^2 + \frac{(1-\alpha)}{6}e(G) \frac{3n^2}{2} \le \frac{\alpha}{2}e(G)n^2 + \frac{(1-\alpha)}{4}e(G) n^2 = \frac{(1+\alpha)}{8}n^3d.$$

Combining this with the fact that the number of good $5$-paths is at least $nd^5/2 - C n^3$ (by \eqref{lower_bound}), we get 
$$\frac{nd^5}{2}  - C n^3 \le  \frac{(1+\alpha)}{8}n^3d,$$

which simplifies to $\frac{d^5}{2}  \le  \frac{(1+\alpha)}{8}n^2d + C n^2 = (1+ o(1)) \frac{(1+\alpha)}{8}n^2d$. Here we used that $d \ge t= \Omega(\sqrt{n})$ (by \eqref{ConnectionBetweendaveandtave}). Therefore, 
\begin{equation}
\label{boundingd}
d \le (1+o(1)) \left(\frac{1+\alpha}{4}\right)^{1/4} \sqrt{n}.
\end{equation}

Recall that when defining $\mathcal D$ we decomposed the edges of each crown-block into a triangle and $2$-paths. This means that the number of triangles of $G$ that belong to crown-blocks of $G$ is at most $\frac{\alpha_1 e(G)}{3}+\frac{\alpha_2e(G)}{2} \le \frac{\alpha e(G)}{2}$, and the number of triangles that belong to $K_4$-blocks of $G$ is at most $\frac{4(1-\alpha)e(G)}{6}$. Therefore, the total number of triangles in $G$ is at most \begin{equation} \label{boundingtriangles} \frac{\alpha e(G)}{2} + \frac{4(1-\alpha)e(G)}{6} = \frac{4-\alpha }{6} e(G) = \frac{(4-\alpha)nd}{12}. \end{equation}
Now using \eqref{boundingd}, we obtain that the number of triangles in $G$ is at most
$$(1+o(1)) \left(\frac{1+\alpha}{4}\right)^{1/4} \frac{(4-\alpha)}{12} n^{3/2}.$$
Now optimizing the coefficient of $n^{3/2}$ over $0 \le \alpha \le 1$, one obtains that it is maximized at $\alpha = 0$, giving the desired upper bound of $(1+o(1)) \frac{1}{3 \sqrt{2}} n^{3/2}.$

\section{On hypergraphs of girth $6$ and further improvement}
\label{Girth6plusImprovementSection}

In this section we will first study $r$-uniform hypergraphs of girth $6$, and prove Theorem \ref{girth6}. Then we use Theorem \ref{girth6} to further (slightly) improve the estimate in Theorem \ref{Main_Result} on the number of triangles in a $C_5$-free graph.

\subsection{Girth $6$ hypergraphs: Proof of Theorem \ref{girth6}}
\label{girth6section}

Let $d$ be the average degree of $H$. Our aim is to show that $d \le \frac{\sqrt{n}}{\sqrt{r}(r-1)}.$
If a vertex has degree less than $d/r$, then we may delete it and the edges incident to it without decreasing the average degree. So we may assume that the minimum degree of $H$, $\delta(H) \ge d/r$.

Suppose there is a vertex $v$ of degree $c\sqrt{n}$ for some constant $c$. Then the first neighborhood $N^H_1(v) := \{x \in V(H) \setminus \{v \}  \mid v,x \in h \text{ for some } h \in E(H)\}$ has size more than $c\sqrt{n}(r-1)$ (since $H$ is linear), and the second neighborhood $ N^H_2(v) = \{x \in V(H) \setminus (N^H_1(v) \cup \{v\}) \mid \exists h \in E(H) \text{ such that } x \in h \text{ and } h \cap N^H_1(v) \not = \emptyset \}$ has size more than $$c\sqrt{n}(r-1) \times \delta(H) (r-1) \ge c\sqrt{n}(r-1) \times \frac{d(r-1)}{r} = \frac{c\sqrt{n}(r-1)^2d}{r}.$$ 
Note that here we used that $H$ has no cycles of length at most four.  On the other hand, since $\abs{N^H_2(v)} \le n$, we have $\frac{c\sqrt{n}(r-1)^2d}{r} \le n$, implying that $d \le \frac{r}{(r-1)^2c} \sqrt{n}$. So if $c > \frac{r^{3/2}}{r-1}$, we have the desired bound on $d$. Thus, we may assume $c \le \frac{r^{3/2}}{r-1}$, which proves that the maximum degree of $H$, $d_{max} \le \frac{r^{3/2}}{r-1} \sqrt{n}$.

Let $\partial H$ denote the $2$-shadow graph of $H$. Let $d^{\partial H}$ and $d^{\partial H}_{max}$ denote the average degree and maximum degree of $\partial H$, respectively. Note that since $H$ is linear,  $d^{\partial H} = (r-1)d$ and $d^{\partial H}_{max} = (r-1)d_{max} \le r^{3/2} \sqrt{n}$. 

We say a $3$-path $v_0v_1v_2v_3$ in $\partial H$ is \emph{bad} if either $\{v_0,v_1,v_2\} \subseteq h$ or $\{v_1,v_2,v_3\} \subseteq h$ for some hyperedge $h \in E(H)$; otherwise it is \emph{good}. 

By the Blakley-Roy inequality the total number of (ordered) $3$-walks in $\partial H$ is at least $n(d^{\partial H})^3$. We claim that at most $3n(d^{\partial H}_{max})^2$ of these $3$-walks are not $3$-paths. Indeed, suppose $v_0v_1v_2v_3$ is a $3$-walk that is not a $3$-path. Then then there exists a repeated vertex $v$ in the walk such that either $v_0 = v_2 =v$ or $v_1 = v_3=v$ or $v_0 = v_3=v$. Since $v$ can be chosen in $n$ ways and the other two vertices of the walk are adjacent to $v$, we can choose them in at most $(d^{\partial H}_{max})^{2}$ different ways. Therefore, the number of (ordered) $3$-paths in $\partial H$ is at least $n(d^{\partial H})^3-3n(d^{\partial H}_{max})^2 \ge n(d^{\partial H})^3 - 3n(r^{3/2} \sqrt{n})^2 = n(d^{\partial H})^3 - 3r^3n^2$.

We will show that most of these $3$-paths are good by bounding the number of bad $3$-paths. Suppose $v_0v_1v_2v_3$ is a bad $3$-path. Then either $\{v_0,v_1,v_2\}$ or $\{v_1,v_2,v_3\}$ is contained  in some hyperedge $h \in E(H)$. In the first case, the number of choices for $v_0v_1v_2$ is $\abs{E(H)}\binom{r}{3}3!$ as there are $\binom{r}{3}$ ways to choose the vertices $v_0,v_1,v_2$ from a hyperedge of $H$ and then $3!$ ways to order them. And there are at most $d^{\partial H}_{max}$ choices for $v_3$. The second case is similar. Therefore, in total, the number of bad $3$-paths in $\partial H$ is at most $2\abs{E(H)}\binom{r}{3}3! d^{\partial H}_{max} < 2\frac{nd}{r}r^3 d^{\partial H}_{max} \le 2n r^2 d_{max} d^{\partial H}_{max} \le  2 \frac{r^{5}}{r-1} n^{2}$. So the number of (ordered) good $3$-paths in $\partial H$ is at least 
\begin{equation}
\label{lower3paths}
n(d^{\partial H})^3 - 3r^3n^2-  2 \frac{r^{5}}{r-1} n^{2} = n(d^{\partial H})^3 - c_r n^2 = (r-1)^3 d^3n - c_r n^2,
\end{equation} where $c_r = 3 r^3 +\frac{2r^5}{r-1}.$

The following claim is useful for upper bounding the number of (ordered) good $3$-paths in $\partial H$.

\begin{claim}
\label{CyclesInShadow}
If $C$ is a cycle of length at most five in $\partial H$, then its vertex set is contained in some hyperedge of $H$.
\end{claim}
\begin{proof}
Let $v_1,v_2,\ldots,v_k,v_1$ be a cycle of length $k$ in $\partial H$ (for some $k \le 5$). For each $i$, let $h_i$ be the hyperedge of $H$ containing $v_i,v_{i+1}$ (addition in the subscripts is taken modulo $k$). If these $k$ hyperedges are not all the same, there exists $j, j'$ such that $h_j, h_{j+1}, \ldots, h_{j'}$ are all distinct but $h_{j'+1} = h_j$. So these hyperedges form a cycle in $H$ of length at most $k \le 5$, a contradiction. Therefore,  $h_1 = h_2 = \ldots =h_k = h$; then $v_1,v_2,\ldots,v_k \in h$, as desired.
\end{proof}

In order to upper bound the number of (ordered) good $3$-paths in $\partial H$, let us first fix a hyperedge $h$ of $H$, and bound the number of good $3$-paths $v_0v_1v_2v_3$ such that $v_0,v_1 \in h$.

\begin{claim}
\label{FixingOneHyperedge}
For any vertex $v \not \in h$, there are at most $(r-1)$ good $3$-paths $v_0v_1v_2v$ such that $v_0,v_1 \in h$. 
\end{claim}
\begin{proof}
Suppose $v_0v_1v_2v$ and $v'_0v'_1v'_2v$ are good $3$-paths with $v_0,v_1,v'_0,v'_1 \in h$. Then $v_2, v'_2 \not \in h$ because it would contradict the definition of a good $3$-path. We will prove that $v_1=v_1'$ and $v_2=v_2'$.

Suppose $v_1 \not = v'_1$. Then depending on whether $v_2 = v'_2$ or not, either $v_1v'_1v'_2vv_2$ forms a five-cycle or $v_1v'_1v'_2$ forms a triangle in $\partial H$. Then by Claim \ref{CyclesInShadow}, $v_1,v'_1,v'_2 \in h'$ for some hyperedge $h' \in E(H)$. (Note that $h' \not = h$, since $v'_2 \not \in h$.) But then $h$ and $h'$ are two different hyperedges of $H$ that share at least two vertices, namely $v_1,v'_1$, contradicting the fact that $H$ is linear. Thus $v_1 = v'_1$.

Now if $v_2 \not = v'_2$, then $vv_2v_1v'_2$ is a four-cycle in $\partial H$, so it must be contained in a hyperedge of $H$, but this means the $3$-path $v_0v_1v_2v$ is bad, a contradiction. Thus $v_2 = v'_2$. 

In summary, any two good $3$-paths $v_0v_1v_2v$ and $v'_0v'_1v'_2v$ with $v_0,v_1,v'_0,v'_1 \in h$ can only differ in their first vertex, of which there are at most $r-1$ choices, proving the claim.
\end{proof}

Claim \ref{FixingOneHyperedge} implies that for any fixed hyperedge $h \in E(H)$, there are at most $(r-1)n$ good $3$-paths $v_0v_1v_2v_3$ with $v_0,v_1 \in h$. Therefore, the total number of good $3$-paths in $H$ is at most $\abs{E(H)} (r-1)n = \frac{(r-1)d n^2}{r}$. 

Combining this with \eqref{lower3paths}, we obtain $(r-1)^3 d^3n - c_r n^2 \le \frac{(r-1)d n^2}{r}$. Dividing through by $d$ and using that $d = \Omega(\sqrt{n})$, we get $(r-1)^3 d^2n  \le  (1+o(1)) \frac{(r-1) n^2}{r}$ and upon simplification  and rearranging, we get $$d \le (1+o(1)) \frac{ \sqrt{n}}{\sqrt{r}(r-1)},$$ so using $\abs{E(H)} = nd/r$, completes the proof.

\subsection{Further improving the estimate on $ex(n,K_3,C_5)$}
\label{FurtherImprovementsection}
Here we slightly improve Theorem \ref{Main_Result}, by establishing a connection to girth $6$ hypergraphs and using Theorem \ref{girth6}. 

Recall that in the proof of Theorem \ref{Main_Result}, $G$ denotes a $C_5$-free graph, and $(1-\alpha) e(G)$ edges of $G$ belong to the $K_4$'s in the edge-decomposition $\mathcal D$ of $G$. Let us note that the vertex sets of two different $K_4$'s of $G$ do not share more than one vertex, since $G$ is $C_5$-free. Consider the $4$-uniform hypergraph $H$ formed by taking the vertex sets of all the $K_4$'s of $G$. Then notice that $H$ is linear and if $H$ contains a (Berge) cycle of length at most $5$, then $G$ contains a $C_5$. Therefore, $H$ is of girth $6$. Therefore, by Theorem \ref{girth6}, $H$ contains at most $n^{3/2}/24$ hyperedges. Thus at most $n^{3/2}/24 \times \binom{4}{2} = n^{3/2}/4$ edges of $G$ belong to the $K_4$'s in the edge-decomposition $\mathcal D$. Therefore, $(1-\alpha) e(G) \le \frac{n^{3/2}}{4}$, which implies 
$d \le \frac{\sqrt{n}}{2(1-\alpha)}.$
Combining this with \eqref{boundingd}, we get $$d \le (1+o(1)) \min \left \{\frac{1}{2(1-\alpha)},  \left(\frac{1+\alpha}{4}\right)^{1/4} \right \} \sqrt{n},$$
so using \eqref{boundingtriangles}, we obtain that the number of triangles in $G$ is at most $$(1+o(1))\frac{(4-\alpha)}{12} \min \left \{ \frac{1}{2(1-\alpha)},  \left(\frac{1+\alpha}{4}\right)^{1/4} \right \} n^{3/2}.$$
The above function is maximized at $\alpha = 0.343171$, proving that $ex(n, K_3, C_5) \le 0.231975 n^{3/2}$.

\section{$C_5$-free and induced-$C_4$-free graphs: Proof of Theorem \ref{indC4C5}}
\label{indC4C5section}

Let $G$ be a $C_5$-free graph on $n$ vertices having no induced copies of $C_4$. Let $G_{\Delta}$ be the subgraph of $G$ consisting of the edges that are contained in triangles of $G$, and let $G_S$ be the subgraph of $G$ consisting of the remaining edges of $G$. Since $G_{\Delta}$ is $C_5$-free and every edge of it is contained in a triangle, by the same argument of the proof of Theorem \ref{Main_Result}, the triangles of $G_{\Delta}$ can be partitioned into crown-blocks and $K_4$-blocks. So there is a decomposition $\mathcal D$ of the edges of $G_{\Delta}$ into paths of length 2, triangles and $K_4$'s. 
First let us note that Claim \ref{atmosttwo2paths} in the proof of Theorem \ref{Main_Result} still holds for $G$ (not just for $G_{\Delta}$), as shown below.

\begin{claim}
	\label{atmosttwo2pathsNew}
Let $u, v$ be two non-adjacent vertices of $G$. Then the number of paths of length $2$ between $u$ and $v$ is at most two. Moreover, if $uxv$ and $uyv$ are the paths of length $2$ between $u$ and $v$, then $x$ and $y$ are adjacent.
\end{claim}
\begin{proof}
The second part of the claim is trivial since $G$ does not contain an induced copy of $C_4$. To see the first part of the claim, suppose $uxv, uyv, uzv$ are three distinct paths of length $2$ in $G$. Then $x$ and $y$ are adjacent, so $uxyvz$ is a $C_5$ in $G$, a contradiction.
\end{proof}

Our goal is to bound the average degree $d$ of $G$. If a vertex has degree less than $d/2$, then it may be deleted without decreasing the average degree of $G$, so we may assume that $G$ has minimum degree at least $d/2$. Now using this fact and Claim \ref{atmosttwo2pathsNew}, one can show that the maximum degree of $G$ is at most $10 \sqrt{n}$ by repeating the same argument as in the proof of Claim \ref{max_degree}. 

We say a $5$-path $v_0v_1v_2v_3v_4v_5$ is \emph{bad} if there exists an $i$ such that $v_iv_{i+1}v_{i+2}$ is a triangle of $G$; otherwise it called \emph{good}. Similarly, a $2$-path $abc$ is \emph{good} if $a$ and $c$ are not adjacent. By the same argument as in the proof of Theorem \ref{Main_Result}, the number of (unordered) good $5$-paths in $G$ is at least \begin{equation}
\label{looower}
\frac{nd^5}{2} - Cn^3
\end{equation} for some constant $C > 0$. Now we bound the number of good $5$-paths in $G$ from above. Let $\abs{E(G_{\Delta})} = \alpha \abs{E(G)}$ for some $\alpha \ge 0$, so $\abs{E(G_S)} = (1-\alpha) \abs{E(G)}$. 

\begin{claim}
\label{single_edge}
The number of good $5$-paths in $G$ whose middle edge is contained in $G_S$ is at most $ \abs{E(G_S)}n^2$.
\end{claim}
\begin{proof}
The proof is very similar to that of the proof of Claim \ref{Claim2path}. A good $5$-path $xyabzw$ whose middle edge $ab$ is in $G_S$ contains good $2$-paths, $xya, bzw$ as subpaths.

Let $n_a$ be the number of good $2$-paths in $G$ of the form $axy$ where $x, y \not = b$, and let $n_b$ be the number of good $2$-paths in $G$ of the form $bxy$ where $x, y \not = a$. 
Then the number of good $5$-paths whose middle edge is $ab$ is at most $ n_an_b \le  (n_a + n_b)^2/4.$ By the same argument as in the proof of Claim \ref{Claim2path}, it is easy to see that $n_a + n_b \le 2n$, so the number of good $5$-paths whose middle edge is $ab \in E(G_S)$ at most $n^2$. Adding these estimates for all the edges $ab \in E(G_S)$  finishes the proof of the claim.
\end{proof}

Let us further assume that the number of edges of $G_{\Delta}$ that belong to paths of length 2, triangles and $K_4$'s in its edge-decomposition $\mathcal D$ be $\alpha_1 \abs{E(G)}, \alpha_2 \abs{E(G)}, \alpha_3 \abs{E(G)}$, respectively. (Of course, $\alpha_1 + \alpha_2 + \alpha_3 = \alpha$.) Since Claim \ref{atmosttwo2pathsNew} holds, one can easily check that the proofs of Claim \ref{Claim2path}, Claim \ref{Triangle} and Claim \ref{pathsinK_4} are still valid, so these claims hold in the current setting too. These claims, together with Claim \ref{single_edge}, imply that the number of good $5$-paths in $G$ is at most $$
    \frac{\alpha_1 \abs{E(G)}}{2} n^2 + \frac{\alpha_2 \abs{E(G)}}{3} \frac{4n^2}{3} + \frac{\alpha_3 \abs{E(G)}}{6} \frac{3n^2}{2} +
\abs{E(G_S)}n^2 \newline \le \frac{\alpha \abs{E(G)}}{2} n^2 + (1- \alpha) \abs{E(G)} n^2.
$$

We will now bound the right-hand-side of the above inequality by carefully selecting a $C_5$-free, and $C_4$-free subgraph $G'$ of $G$, as follows: We select all the edges of $G_S$ and the following edges from $G_{\Delta}$: From each crown-block $\{abc_1, abc_2,\ldots, abc_k\}$ of $G_{\Delta}$, we select the edges $ac_1, ac_2,\ldots,ac_k$ to be in $G'$. From each $K_4$-block $abcd$ we select the edges $ab, bc, ac, ad$ to be in $G'$. 

To show that $G'$ is $C_4$-free we use the following claim.

\begin{claim}
	\label{C_4_inside_block}
All four edges of any $C_4$ in $G$ belong to only one block of $G_{\Delta}$.
\end{claim}

\begin{proof}
Let $xyzw$ be a $4$-cycle in $G$. Then since $G$ does not contain an induced copy of $C_4$, either $xz$ or $yw$ is an edge of $G$. In the first case, $xzy, xzw$ are triangles of $G$, and in the second case $ywz, ywx$ are triangles of $G$. In both cases, the two triangles share an edge, so they belong to the same block of $G_{\Delta}$. Hence, all four edges of $xyzw$ lie in the same block of $G_{\Delta}$.
\end{proof}

By Claim \ref{C_4_inside_block}, the edge set of every $C_4$ is completely contained in some block of $G_{\Delta}$, and it is easy to check that the selected edges in each block of $G_{\Delta}$ form a $C_4$-free graph. Therefore, $G'$ is $C_4$-free. Since it is a subgraph of $G$, it is also $C_5$-free. Therefore, by a theorem of Erd\H{o}s and Simonovits \cite{Erd_Sim}, $\abs{E(G')} \le \frac{1}{2 \sqrt{2}} n^{3/2}$. On the other hand, since all the edges of $G_S$ and at least half the edges of $G_{\Delta}$ are selected, we have $\abs{E(G')} \ge \abs{E(G_S)} + \frac{\abs{E(G_{\Delta})}}{2} = (1- \alpha) \abs{E(G)} + \frac{\alpha \abs{E(G)}}{2}$. Therefore, 
$$ \frac{\alpha \abs{E(G)}}{2} + (1- \alpha) \abs{E(G)} \le \frac{1}{2 \sqrt{2}} n^{3/2}.$$

Therefore, by the discussion above, the number of good $5$-paths in $G$ is at most $\frac{1}{2 \sqrt{2}} n^{3/2} \times n^2 = \frac{1}{2 \sqrt{2}} n^{7/2}$. Combining this with \eqref{looower}, we get 
$$\frac{nd^5}{2} - Cn^3 \le \frac{1}{2 \sqrt{2}} n^{7/2},$$ so $\frac{nd^5}{2} \le (1 + o(1)) \frac{1}{2 \sqrt{2}} n^{7/2}$, implying that $d \le \frac{\sqrt{n}}{\sqrt[10]{2}}$, finishing the proof.

\section*{Acknowledgements}
The research of the authors is partially supported by the National Research, Development and Innovation Office -- NKFIH, grant K116769.


\begin{thebibliography}{10}
		
		\bibitem{AlonS} N. Alon and C. Shikhelman, Many $T$ copies in $H$-free graphs. \emph{Journal of Combinatorial Theory, Series B} \textbf{121} (2016) 146--172.
		
		\bibitem{BolGy} B. Bollob\'as and E. Gy\H ori, Pentagons vs. triangles. \emph{Discrete Mathematics} \textbf{308.19} (2008) 4332--4336.
		
		\bibitem{Erd} P. Erd\H os, On some problems in graph theory, combinatorial analysis and combinatorial number theory.  B. Bollobas (Ed.), Graph Theory and Combinatorics (Cambridge, 1983), Academic Press, London (1984) 1--17.
		
		\bibitem{Erd_Sim} P. Erd\H os  and M. Simonovits, Compactness results in extremal graph theory. \emph{Combinatorica} \textbf{2.3} (1982) 275--288.
		
		\bibitem{EGM} B. Ergemlidze, E. Gy\H ori and A. Methuku. Tur\'an number of an induced complete bipartite graph plus an odd cycle. \textit{Combinatorics, Probability and Computing}, (2018) 1--12.
		
		\bibitem{ErgGMS} B. Ergemlidze, E. Gy\H ori, A. Methuku and N. Salia. A note on the maximum number of triangles in a $C_5$-free graph.	\textit{Journal of Graph Theory}, (2018) 1--4, In print.
		
		\bibitem{Furedi} Z. F\"uredi and L. \"Ozkahya. On 3-uniform hypergraphs without a cycle of a given length. \textit{Discrete Applied Mathematics}, \textbf{216}, (2017) 582--588.
		
	\bibitem{Lazeb_Verstraete}
F. Lazebnik and J. Verstra\"ete. On hypergraphs of girth five. \textit{Electron. J. Combin} 10 (2003) R25.
		
	\bibitem{Loh_Tait_Timmons} P. Loh, M. Tait, C. Timmons and R.M. Zhou.  Induced Tur\'an numbers. \emph{Combinatorics, Probability and Computing}, \textbf{27(2)} (2017) 274--288.
	
		
	\end{thebibliography}
\end{document}